%% file: posdef.tex
\documentclass[12pt]{amsart}

\include{mypackages}

\include{mymacros}
\usepackage{diagrams}
\usepackage[all]{xy}
\usepackage{lineno}

\newcommand{\aGNH}{\mathcal{GNH}}

\DeclareMathOperator{\cyc}{cyc}

\begin{document}

\title[Deformations of nilHecke algebras]{Deformations of nilHecke Algebras}
\author[Benjamin Cooper]{Benjamin Cooper}
\address{Institut f\"{u}r Mathematik, Universit\"{a}t Z\"{u}rich, Winterthurerstrasse 190, CH-8057 Z\"{u}rich}
\email{benjamin.cooper\char 64 math.uzh.ch}

\begin{abstract}
  Presentations are given for the deformed nilHecke algebras which appear in
  the work of Bressler and Evens. The obstruction to the braid relation
  being satisfied by each algebra is given. A family of positively graded
  deformations of nilHecke algebras is defined and shown to be independent
  of the deformations determined by oriented cohomology theories.
\end{abstract}

\maketitle

\section{Introduction}\label{introsec}
The nilHecke algebras $\NH_n$ have found essential applications to a variety
of fields: from numerical analysis, to the study of the topology and the
representation theory of the group $\SL_n(\CC)$.  More recently, the program
of categorification developed by Khovanov, Lauda and Rouquier \cite{KhL, R}
has found that the nilHecke algebras are uniquely important in the study of
geometric representation theory \cite{R, VV}. Other authors have found that
differential graded analogues of nilHecke algebras control a codimension 2
extension of the Heegaard-Floer topological field theory \cite{LOT}. The
many places in which these algebras appear suggests that one should study of
the many shapes that they can form. In this article, families of
deformations of nilHecke algebras are introduced and studied.

The origins of the nilHecke algebra are rooted in the development of
Schubert calculus for the space of full flags:
$$\Fl_n = \{ V_0 \subset V_1 \subset \cdots \subset V_{n-1} \subset V_n = \CC^n \} \conj{ where } \dim V_i = i.$$
Borel and Hirzebruch showed that cohomology of this space has an invariant
description in terms of Chern classes: $H^*(\Fl_n) \cong \ZZ[x_1,\ldots,
x_n]_{\Sigma_n}$ \cite{Borel}.  On the other hand, this cohomology ring
inherits a basis from a stratification by Schubert cells.  In \cite{BGG, D},
convolutions $H^*(\Fl_n) \to H^{*-2}(\Fl_n)$ were found which allow one to
interpolate between these two descriptions of cohomology: expressing the
classes which are dual to Schubert cells in terms of Chern classes. These
convolutions generate the most important part of the nilHecke algebra.

In \cite{BE1, BE2}, Bressler and Evens studied how the solution to the
problem of describing the Schubert basis changes as one varies the
cohomology theory $H^*(\Fl_n)$ through complex oriented cohomology theories
$h^*(\Fl_n)$.  Computing the convolutions leads to analogues of divided
difference operators which depend on a formal group law. They proved that
these operators continue to interpolate between the two descriptions of the
generalized cohomology $h^*(\Fl_n)$ of the space of flags.

In this article, we study the algebra of the divided difference operators
defined by Bressler and Evens. Each complex oriented cohomology theory
determines a nilHecke algebra. In special cases, this algebra is a
deformation of the nilHecke algebra. Theorem \ref{bsgnhthm} derives a
presentation for each such algebra. The main result in \cite{BE1} is that
the braid relation is satisfied only by divided difference operators which
have been deformed by polynomial formal group laws. The presentation
contained here gives a precise obstruction to the braid relation in terms of
two power series: $l,r\in R[[x,y,z]]$. This allows us to isolate
deformations of intermediate complexity with respect to the braid relation,
see Remark \ref{stackrem}.

The deformations determined by complex oriented cohomology theories must be
restricted to those with negatively graded deformation parameters.  In
Section \ref{defpossec}, a family of deformations which have positively
graded deformation parameters are studied. These algebras cannot be derived
from complex oriented cohomology theories.

This paper is part of an ongoing research project \cite{Cooper,BelC}.

\subsection*{Acknowledgements}
The author would like thank M. Khovanov for discussing these results in
2011.  Also A. Beliakova, M. Hill, N. Kuhn, A. Lauda and W. Wang for their
interest.

In 2012, these results were presented at the University of Virginia,
University of Zurich, the AMS Tampa meeting and Knots in Washington.  The
author would like to thank O. Yacobi and A. Hoffnung for the opportunity to
present these results at the Fields Institute in May 2012.

Special thanks to the organizers of the conference ``Quantum Topology and
Hyperbolic Geometry'' for their encouragement of younger researchers.

\section{Complex Oriented Cohomology Theories}\label{complexorthsec}
In this section we recall some of the prerequisite ideas from topology.

\subsection{Generalized cohomology}\label{generalizedcohomology}
Here we discuss a class of cohomology theories which support a theory of
Schubert calculus.

\begin{definition}
A \emph{cohomology theory} $h^* = \{h^n\}_{n\in\ZZ}$ is a family of
contravariant functors from the homotopy category of topological spaces to
the category of abelian groups which satisfy the Eilenberg-Steenrod
axioms. 
\end{definition}

\begin{definition}
When the collection
  $h^*(X)$ forms a graded ring, the cohomology theory $h^*$ is called
  \emph{multiplicative}. A multiplicative cohomology theory $h^*$ is
  \emph{oriented} by $t\in h^*(\CC P^\infty)$ when $i^*(t)$ generates
  $h^*(\CC P^1)$ as an $h^*(pt)$-module where $i : \CC P^1\hookrightarrow
  \CC P^\infty$ is the standard inclusion.

  In what follows all cohomology theories will be multiplicative complex
  oriented cohomology theories unless otherwise noted.
\end{definition}

\begin{remark}
For such a cohomology theory $h^*$, the collapsing of the Atiyah-Hirzebruch spectral sequence 
gives an isomorphism,
$$h^*(BT^n) \cong h^*(pt)[[x_1,\ldots, x_n]].$$
A splitting principle argument yields the isomorphism,
$$h^*(BU(n)) \cong h^*(BT^n)^{\Sigma_n} \cong h^*(pt)[[x_1,\ldots, x_n]]^{\Sigma_n} \cong h^*(pt)[[c_1,\ldots, c_n]].$$
where $\Sigma_n$ is the symmetric group. This results in a theory of Chern classes:
if $h^*$ is oriented by $x_h\in h^*(\CC P^\infty)$ then to any
complex vector bundle $K\to X$ one can assign classes $c_i(K)\in h^{2i}(X)$
which satisfy the following properties:
\begin{enumerate}
\item $c_0(K) = 1$.
\item Naturality: $c_i(f^*K) = f^*c_i(K)$ for all maps $f : Y\to X$.
\item Whitney sum:
$$c_n(K\oplus K') = \sum_{i+j = n} c_i(K)\cdot c_j(K')$$
\item Let $\gamma \to \CC P^{\infty}$ be the Hopf line bundle over $\CC P^{\infty}$ then $c_1(\gamma) = x_h$, when $h^*$ is oriented by $x_h\in h^*(\CC P^\infty)$.
\end{enumerate}
\end{remark}

\begin{remark}
Properties (1)--(4) hold uniformly with respect to the cohomology theories
under consideration. There is one property which changes as we change the
cohomology theory, if $L$ and $L'$ are line bundles over a space $X$ then
the first Chern class of their tensor product $c_1(L \otimes L')$ depends on
$h^*$:
$$c_1(L \otimes L') = F(c_1(L), c_1(L')),$$
where $F(x,y) = \otimes^*(x_h) \in h^*(\CC P^\infty\times \CC P^\infty) \cong
h^*(pt)[[x,y]]$.
\end{remark}

\begin{definition}\label{cofgldef}
  The power series $F(x,y)\in h^*(pt)[[x,y]]$ is the \emph{formal group law
    associated to} the cohomology theory $h^*$.
\end{definition}

\begin{definition}\label{fgldef}
A \emph{formal group law} is a commutative ring $R$ together with an element $F(x,y)\in
R[[x,y]]$ such that
\begin{enumerate}
\item $F(0,x) = x$ and $F(x,0) = x$
\item $F(x,y) = F(y,x)$
\item $F(x,F(y,z)) = F(F(x,y),z)$
\end{enumerate}
for all $x,y,z\in R$.

A formal group law is \emph{graded} when $R$ is a graded commutative ring,
$|x| = 2$, $|y| = 2$ and the series $F(x,y)\in R[[x,y]]$ is homogeneous of degree 2.
\end{definition}

\begin{remark}
  The formal group law associated to a cohomology theory in Definition
  \ref{cofgldef} satisfies these properties because the tensor product of
  line bundles is unital, commutative and associative.
\end{remark}

\begin{notation}\label{notinv}
  By the formal implicit function theorem, there exists a unique series
  $\inv(x) \in R[[x]]$ such that $F(x, \inv(x)) = 0$, see
  \cite{Hazewinkel}. Throughout the remainder of the paper we will use the
  notation,
$$x +_F y = F(x,y)\quad\normaltext{  and  }\quad x-_F y = F(x,\inv(y)).$$  
\end{notation}

\begin{definition}\label{symdef}
A formal group law $F$ is \emph{symmetric} when it satisfies
$$x_i -_F x_{i+1} = -(x_{i+1} -_F x_i).$$
\end{definition}

\begin{definition} 
A \emph{homomorphism} $\phi : F \to G$ between two formal group laws $F,G \in R[[x,y]]$ is a power series $\phi \in R[[x]]$ which satisfies the equation:
$$\phi(F(x,y)) = G(\phi(x),\phi(y)).$$

Suppose that $F \in R[[x,y]]$ and $G\in S[[x,y]]$ are formal group laws:
$$F(x,y) = \sum_{i,j} c_{ij} x^i y^j \quad\normaltext{  and  }\quad G(x,y) = \sum_{i,j} d_{ij} x^i y^j.$$
If $\varphi : R\to S$ is a ring homomorphism such that $d_{ij} =
\varphi(c_{ij})$ then we write $\varphi(F) = G$. 
\end{definition}

\begin{definition}
A formal group law $F_U\in \LL[[x,y]]$ is \emph{universal} when,
for all formal group laws $G\in S[[x,y]]$, there exists a ring
homomorphism $\varphi : \LL \to S$ such that $\varphi(F_U) = G$.
\end{definition}

\begin{theorem}{(Lazard)}\label{lazardthm}
There exists a universal formal group law $F_U\in \LL[[x,y]]$ and the ring of coefficients
$\LL$ is a polynomial ring on infinitely many generators, $\{u_i\}_{i\in\ZZ_+}$,
$$\LL \cong \ZZ[u_1,u_2,u_3,\ldots]$$
where $|u_i| = -2i$.
\end{theorem}

\subsection{Complex cobordism and genera}\label{complexcomplexsec}
In this section we recall the cohomology theory complex cobordism. For more
details see \cite{Adams}.

\begin{definition}
  There is a cohomology theory $MU^*$ which is defined so that for a
  manifold $X$ of dimension $r$, $MU^{r-l}(X)$ is spanned by cobordism
  classes of maps $f: M\to X$ where $f$ is smooth, proper and $M$ is an
  $l$-dimensional manifold equipped with a stable almost complex structure.
  If $s\in MU^*(X)$, $s : M \to X$ and $g : Y \to X$ then $g^*(s) =
  M\times_X Y \to Y \in MU^*(Y)$.

If $L\to X$ is a complex line bundle over $X$ then the zero section of the
Thom space $X^L$, $s : X \to X^L$ determines a class $s \in MU^2(X^L)$ and
$MU^*$ is oriented by the choice $c_1(L) = s^*(s) \in MU^2(X)$.
\end{definition}

\begin{remark}
  The ring $MU^*(pt)$ consists of compact stably almost complex manifolds
  modulo the cobordism relation. 
\end{remark}

\begin{theorem}{(\cite{Quillen})}\label{quillenthm}
Complex cobordism of a point is isomorphic to the Lazard ring:
$$MU^*(pt) \cong \LL.$$
\end{theorem}

\begin{remark}
The induced isomorphism,
$$\Morph(\LL,R)\cong \Morph(MU^*(pt), R)$$
implies that the set of formal group laws in $R$ is in bijection with the set
of \emph{$R$-genera} or $R$-valued cobordism invariants of complex manifolds. 
\end{remark}

\begin{theorem}{(\cite{Thom})}\label{thmthm}
The coefficient ring of complex cobordism is the polynomial ring on complex projective spaces,
$$MU^*(pt)\otimes \QQ \cong \QQ[\CC P^1, \CC P^2,\ldots] \conj{with} \vert\CC P^n\vert = -2i.$$
\end{theorem}

\begin{corollary}
  A genus $\rho : MU^*(pt)\to R$ is determined by its values on complex
  projective spaces.
\end{corollary}

\begin{remark}
 The generators $u_n$ of $\LL$ in Theorem \ref{lazardthm}
  are not equal to complex projective spaces $\CC P^n$ in Theorem \ref{thmthm}.
\end{remark}

\begin{remark}
Given a genus $\rho : MU^*(pt) \to R$, one can hope to define a cohomology
theory $h^*$ with genus $\rho$ by extending coefficients:
$$h^*(X) = MU^*(X)\otimes_{MU^*(pt)} R.$$
This determines a cohomology theory $h^*$ when the ring $R$ is flat as an
$MU^*(pt)$-module, or more generally, when the criteria of the Landweber
exactness theorem hold, see \cite{Landweber}.
\end{remark}

Below are examples of genera and their formal group laws. Not all of these examples correspond to cohomology theories.

\begin{examples}{(Formal group laws and genera)}\label{examplesofgenera}
\begin{enumerate}
\item The formal group law associated to cohomology is the additive group law $F_H(x,y) = x+ y$.
The associated genus $\rho : MU^*(pt)\to \ZZ$ is determined by $\rho(\CC P^{n}) = 0$.

\item The graded Todd genus, determined by $\Td(\CC P^n) = \beta^n \in
  \ZZ[\beta]$, is associated to the formal group law $F_K(x,y) = x + y - \beta
  xy$ for connective complex K-theory. The ring $\ZZ[\beta]$ is graded with $\beta$ placed in degree $-2$. 

Complex K-theory has the same formal group law as connective complex K-theory, but is defined over the ground ring $\ZZ[\beta,\beta^{-1}]$.

\item Hirzebruch's $\chi_a$-genus is defined using Dolbeault cohomology of $M$:
$$\chi_a(M) = \sum_{p,q} (-1)^p a^q \dim H^{p,q}(M) \in \mathbb{Z}[a].$$
If $M$ is $\CC P^n$ then $\chi_a(\CC P^n)$ is the cyclotomic polynomial $1-a+a^2-\cdots\pm a^n$. The associated formal group law is given by
$$F_{\chi_a}(x,y) = \frac{x+y + (a-1)xy}{1 + a xy}.$$
This formal group law doesn't quite satisfy our grading conventions. We will instead consider the more general formal group law,
$$F_{\chi_{\alpha,\beta}}(x,y) = \frac{x + y - \beta xy}{1 + \alpha xy} \in \ZZ[\alpha,\beta][[x,y]].$$
The grading is obtained by placing $\alpha$ in degree $-4$ and $\beta$ in degree $-2$.

\item There is a formal group law associated to any 1-dimensional algebraic
  group. If $Q(x) = 1- 2\delta x^2 + \epsilon x^4$ then Euler determined a
  two parameter family of formal group laws
$$F(x,y) = \frac{x \sqrt{Q(y)} + y \sqrt{Q(x)}}{1-\epsilon x^2 y^2}$$
associated to Jacobi quartics.  The power series $F$ is defined over the ring
$\ZZ[\frac{1}{2},\epsilon, \delta]$ in which $\delta$ has degree $-4$ and
$\epsilon$ has degree $-8$. The values of the associated genus $\rho$ are
determined by power series $\log_F(x)$ where
$$\log_F(x) = \int^x_0 \frac{dy}{\sqrt{Q(y)}} \quad\normaltext{ and }\quad \rho(\CC P^n) = \frac{\log_F^{(n+1)}(0)}{n!}.$$
\end{enumerate}
\end{examples}

\section{nilHecke algebras}\label{nhpresentationssec}

In \cite{BE1, BE2}, the authors Bressler and Evens found that each formal
group law yields a deformed version the divided difference operators
$\partial_j$ which determine the nilHecke algebra. In this section, we
derive a presentation for nilHecke algebras $\GNH_{n}$ associated to each
formal group law $F\in R[[x,y]]$.

We recall the definition of the nilHecke algebra below in order to provide a
framework for comparison.

\begin{definition}\label{nilheckedef}
The \emph{nilHecke algebra} $\NH_n$ is the graded ring generated by operators $x_i$
in degree $2$, $1\leq i \leq n$, and $\partial_j$ in degree $-2$, $1 \leq j
< n$, subject to the relations:
\[
 \begin{array}{ll}
  \partial_i^2 = 0,  & \partial_i\partial_{i+1}\partial_i = \partial_{i+1}\partial_i\partial_{i+1}\\
   x_i \partial_i - \partial_i x_{i+1} = 1, &  \partial_i x_i - x_{i+1} \partial_i = 1.
  \end{array}
\]
The operators also satisfy the far commutativity relations,
\[
\begin{array}{ll}
\partial_i x_j = x_j\partial_i  \quad\text{if $|i-j|>1$}, &   \partial_i\partial_j = \partial_j\partial_i \quad \text{if $|i-j|>1$},\\   &  x_i x_j = x_j x_i \quad\text{for $1\leq i,j\leq n$}.
\end{array}
\]
\end{definition}

\begin{definition}\label{divdiffdef}
The nilHecke algebra $\NH_n$ acts on the polynomial ring $\aP_n = H^*(BT^n;\ZZ)\cong
\ZZ[x_1,\ldots,x_n]$ by letting $x_i$ act by multiplication and $\partial_j$ act on $f\in\aP_n$ by the \emph{divided difference operator}:
\begin{equation}\label{divdiffeq}
\partial_j(f) = (\Id + \s_j)\frac{1}{x_j - x_{j+1}} = \frac{f-\s_j(f)}{x_j - x_{j+1}}
\end{equation}
where $\s_j$ is the transposition $(j,j+1)$.
\end{definition}

\begin{remark}\label{divdiffrem}
  The image and the kernel of the divided difference operator $\partial_j$ are contained in the
  subring of the polynomial ring consisting of those polynomials which are
  symmetric in the variables $x_j$ and $x_{j+1}$:
$$\img(\partial_j)\subset \aP_n^{(j,j+1)}\conj{and} \aP_n^{(j,j+1)}\subset \ker(\partial_j).$$
It follows that the subring of symmetric polynomials $\Sym_n=\aP_n^{\Sigma_n}$ is contained in the kernel of each operator:
$$\Sym_n\subset\ker(\partial_j)\conj{for} 1\leq j < n.$$
Since each divided difference operator satisfies the Leibniz rule,
\begin{equation}\label{leibeq}
\partial_j(fg) = \partial_j(f)g + \s_j(f)\partial_j(g),
\end{equation}
it follows that $\partial_j : \aP^n\to\aP^n$ is $\Sym_n$-equivariant for all $1\leq j < n$. 
\end{remark}

If $F \in R[[x,y]]$ is a graded formal group law then the expression
$x_i -_F x_{i+1}$ denotes the power series $F(x_i, \inv(x_{i+1}))\in
R[[x_i,x_{i+1}]]$, see Notation \ref{notinv}.  This power series has a
multiplicative inverse $1/(x_i -_F x_{i+1})$ in the ring $R[[x,y]]$ when $R$
has no 2-torsion, see \cite{BE1}. 

The definition below is implicit in the work of Bressler and Evens \cite{BE1,BE2}.

\begin{definition}\label{begnhdef}
  Suppose that $h^*$ is an oriented cohomology theory and let $F\in
  R[[x,y]]$ be the graded formal group law associated to $h^*$. If $\aP_n =
  h^*(BT^n) \cong h^*(pt)[[x_1,\ldots,x_n]]$ is the ring of power series
  with coefficients in $h^*(pt)$ then the \emph{generalized nilHecke
    algebra} $\GNH_n$ is the algebra generated by the operators:
$$x_i : \aP_n \to \aP_n \conj{ and } \partial_j : \aP_n \to \aP_n$$
acting on $\aP_n$. A polynomial generator $x_i$ acts on $\aP_n$ by multiplication as in Definition \ref{nilheckedef} and the divided difference operator $\partial_j$ acts on an element $f\in\aP_n$ by
\begin{equation}\label{bediv}
\partial_j = (\Id + \s_j)\frac{1}{x_j -_F x_{j+1}}.
\end{equation}
\end{definition}

\begin{remark}
  Compare Equation \eqref{bediv} to Equation \eqref{divdiffeq}. The former
  is defined over a ring of power series and the latter is defined over a
  polynomial ring. Also Equation \eqref{bediv} is defined over an extension
  by $h^*(pt)$.
\end{remark}

\begin{proposition}\label{symprop}
For each graded formal group law $F\in R[[x,y]]$ the operators \eqref{bediv} satisfy the following properties:
\begin{enumerate}
\item If $f\in \aP_n$ is a polynomial then $\partial_j(f)\in \aP_n^{(j,j+1)}$ is a series which is symmetric in the variables $x_j$ and $x_{j+1}$.
\item If $f\in \aP_n^{(j,j+1)}$ is a series which is symmetric in the variables $x_j$ and $x_{j+1}$ and $g\in \aP_n$ is a series then
$$\partial_j(fg) = f \partial_j(g).$$
\end{enumerate}
\end{proposition}

The proof of the proposition stems from an examination of the operators
\eqref{bediv}. It implies that the action of the algebra defined by Bressler
and Evens on the ring $\aP_n$ is $\Sym_n$-equivariant; compare to Remark
\ref{divdiffrem}. When $F$ is a symmetric formal group law all of the
properties discussed in Remark \ref{divdiffrem} continue to hold.

The definition below fixes some variables which will be used in Definition \ref{gnhdef}.

\begin{definition}
Suppose that $F\in R[[x,y]]$ is a graded formal group law. Then the power series $s,t\in R[[x_i,x_{i+1}]]$  are defined to be
$$s(x_i,x_{i+1}) = \frac{1}{x_i -_F x_{i+1}} + \frac{1}{x_{i+1} -_F x_i} \normaltext{\quad and\quad} t(x_i,x_{i+1}) = \frac{x_i - x_{i+1}}{x_i -_F x_{i+1}}.$$
Let $l,r \in R[[x_i,x_{i+1},x_{i+2}]]$ be the power series 
{\Small $$l=\frac{1}{(x_{i+2}-_F x_{i+1})(x_{i+2}-_F x_{i})} - \frac{1}{(x_{i}-_F x_{i+1})(x_{i+2}-_F x_{i})} - \frac{1}{(x_{i+1}-_F x_{i})(x_{i+2}-_F x_{i+1})}$$}
{\Small $$r=\frac{1}{(x_{i+1}-_F x_{i+2})(x_{i+2}-_F x_{i})} + \frac{1}{(x_{i+2}-_F x_{i+1})(x_{i+1} -_F x_{i})} - \frac{1}{(x_{i+1}-_F x_{i})(x_{i+2}-_F x_{i})}.$$}
\end{definition}

\begin{remark}
  The elements $s$, $t$, $l$ and $r$ are homogeneous of degree $-2$, $0$,
$-4$ and $-4$ respectively.
\end{remark}

\begin{definition}\label{gnhdef}
  If $F \in R[[x,y]]$ is a graded formal group law then the
  \emph{generalized nilHecke algebra} $\GNH_n$ associated to 
  $F$ is the graded $R$-algebra generated by operators $x_i$ in degree
  $2$, $1 \leq i \leq n$ and $\partial_j$ in degree $-2$, $1\leq j < n$,
  subject to the relations enumerated below.

\begin{enumerate}
\item The two quadratic relations are deformed by the $s$ and $t$ functions,
\[ 
\begin{array}{ll}
\partial_i^2 = s(x_i,x_{i+1})\partial_i, &  x_i \partial_i - \partial_i x_{i+1} =  t(x_i,x_{i+1})\cdot 1, \\
 & \partial_i x_i - x_{i+1} \partial_i  = t(x_i,x_{i+1})\cdot 1.  
\end{array}
\]
\item The braid relation is deformed by the $l$ and $r$ functions,
\[
\partial_{i+1}\partial_i\partial_{i+1} - \partial_i\partial_{i+1}\partial_i  =  l(x_i,x_{i+1},x_{i+2})\partial_i + r(x_i,x_{i+1},x_{i+2})\partial_{i+1} .
\]
\item The operators also satisfy far commutativity relations,
\[
\begin{array}{ll}
\partial_i x_j = x_j\partial_i  \quad\text{if $|i-j|>1$}, &   \partial_i\partial_j = \partial_j\partial_i \quad \text{if $|i-j|>1$},\\   &  x_i x_j = x_j x_i \quad\text{for $1\leq i,j\leq n$}.
\end{array}
\]
\end{enumerate}
\noindent All of the relations above are homogenous with respect to the grading.
\end{definition}

\begin{remark}\label{symmfglgnh}
  When the formal group law $F$ in Definition \ref{gnhdef} is symmetric
  (see Definition \ref{symmfglgnh}), the first two relations become
\[
 \begin{array}{lcl}
  \partial_i^2 = 0 & \quad\text{ and }\quad  & \partial_{i+1}\partial_i\partial_{i+1} - \partial_i\partial_{i+1}\partial_i = l(x_i,x_{i+1},x_{i+2})(\partial_{i}  -\partial_{i+1})
  \end{array}
\]
because $r(x_i,x_{i+1},x_{i+2}) = - l(x_i, x_{i+1},x_{i+2})$.
\end{remark}

It is common to represent Hecke algebras using pictures.  In the
illustration below, the generators $x_i$ are represented by a dot on the
$i$th horizontal strand appearing in a picture. The generators $\partial_j$
are represented by a singular crossing between the $j$th and $j+1$st
vertical strands.

{\Small \begin{equation*}
\vcenter
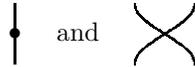
{\xy
    (0,-4)*{};(0,4)*{} **\dir{-}?(1)*\dir{}?(.5)*{\bullet};
\endxy} \quad \text{ and } \quad
\vcenter{\xy 
    (-4,12)*{};(4,20)*{} **\crv{(-4,15) & (4,17)}?(1)*\dir{};
    (4,12)*{};(-4,20)*{} **\crv{(4,15) & (-4,17)}?(1)*\dir{};
\endxy}\;\; 
\end{equation*}}

\noindent
Each dot is graded in degree $2$ and each crossing in degree $-2$. Vertical
strands are indexed by the order in which they appear when reading from left
to right. Following this convention allows us to illustrate the relations
in Definiton \ref{gnhdef} above.

{\Small \begin{equation*}
  \vcenter{\xy 
    (-4,-4)*{};(4,4)*{} **\crv{(-4,-1) & (4,1)}?(1)*\dir{};
    (4,-4)*{};(-4,4)*{} **\crv{(4,-1) & (-4,1)}?(1)*\dir{};
    (-4,4)*{};(4,12)*{} **\crv{(-4,7) & (4,9)}?(1)*\dir{};
    (4,4)*{};(-4,12)*{} **\crv{(4,7) & (-4,9)}?(1)*\dir{};
 \endxy}
\quad =\quad s(x_i,x_{i+1}) \quad \xy
    (-4,-4)*{};(4,4)*{} **\crv{(-4,-1) & (4,1)}?(1)*\dir{};
    (4,-4)*{};(-4,4)*{} **\crv{(4,-1) & (-4,1)}?(1)*\dir{};
    (4,4);(4,8) **\dir{-}?(1)*\dir{};
    (-4,4);(-4,8) **\dir{-}?(1)*\dir{};
    (4,-6);(4,-4) **\dir{-}?(0)*\dir{-};
    (-4,-6);(-4,-4) **\dir{-}?(0)*\dir{-};
 \endxy 
\end{equation*}}

{\Small \begin{eqnarray*}
\xy
  (0,0)*{\xybox{
    (-4,-4)*{};(4,4)*{} **\crv{(-4,-1) & (4,1)}?(1)*\dir{}?(.25)*{\bullet};
    (4,-4)*{};(-4,4)*{} **\crv{(4,-1) & (-4,1)}?(1)*\dir{};
     (8,1)*{};
     (-10,0)*{};(10,0)*{};
     }};
  \endxy
 \; -
 \xy
  (0,0)*{\xybox{
    (-4,-4)*{};(4,4)*{} **\crv{(-4,-1) & (4,1)}?(1)*\dir{}?(.75)*{\bullet};
    (4,-4)*{};(-4,4)*{} **\crv{(4,-1) & (-4,1)}?(1)*\dir{};
     (8,1)*{};
     (-10,0)*{};(10,0)*{};
     }};
  \endxy
 \; = \;\;\; t(x_i, x_{i+1})\;
  \xy
  (4,4);(4,-4) **\dir{-}?(0)*\dir{<}+(2.3,0)*{};
  (-4,4);(-4,-4) **\dir{-}?(0)*\dir{<}+(2.3,0)*{};
  (9,2)*{};
 \endxy
 = 
\xy
  (0,0)*{\xybox{
    (-4,-4)*{};(4,4)*{} **\crv{(-4,-1) & (4,1)}?(1)*\dir{};
    (4,-4)*{};(-4,4)*{} **\crv{(4,-1) & (-4,1)}?(1)*\dir{}?(.75)*{\bullet};
     (8,1)*{ };
     (-10,0)*{};(10,0)*{};
     }};
  \endxy
 \; -
  \xy
  (0,0)*{\xybox{
    (-4,-4)*{};(4,4)*{} **\crv{(-4,-1) & (4,1)}?(1)*\dir{} ;
    (4,-4)*{};(-4,4)*{} **\crv{(4,-1) & (-4,1)}?(1)*\dir{}?(.25)*{\bullet};
     (8,1)*{ };
     (-10,0)*{};(10,0)*{};
     }};
  \endxy 
\end{eqnarray*}}

{\Small \begin{equation*}
 \vcenter{
 \xy 
    (4,-4)*{};(-12,12)*{} **\crv{(4,-1) & (-12,9)}?(1)*\dir{}; 
    (-4,-4)*{};(4,4)*{} **\crv{(-4,-1) & (4,1)}?(1)*\dir{};
    (-12,4)*{};(4,20)*{} **\crv{(-12,7) & (4,17)}?(1)*\dir{}; 
    (4,12)*{};(-4,20)*{} **\crv{(4,15) & (-4,17)}?(1)*\dir{};
    (4,4)*{}; (4,12) **\dir{-};
    (-12,-4)*{}; (-12,4) **\dir{-};
    (-12,12)*{}; (-12,20) **\dir{-}?(1)*\dir{};
\endxy}
 \;\; -\;\;
 \vcenter{
 \xy 
    (-4,-4)*{};(12,12)*{} **\crv{(-4,-1) & (12,9)}?(1)*\dir{};
    (4,-4)*{};(-4,4)*{} **\crv{(4,-1) & (-4,1)}?(1)*\dir{};
    (12,4)*{};(-4,20)*{} **\crv{(12,7) & (-4,17)}?(1)*\dir{};
    (-4,12)*{};(4,20)*{} **\crv{(-4,15) & (4,17)}?(1)*\dir{};
    (-4,4)*{}; (-4,12) **\dir{-};
    (12,-4)*{}; (12,4) **\dir{-};
    (12,12)*{}; (12,20) **\dir{-}?(1)*\dir{};;
\endxy}
= \;\;l(x_i,x_{i+1},x_{i+2})\;\;
\vcenter{
\xy
    (-4,4)*{};(4,12)*{} **\crv{(-4,7) & (4,9)}?(1);
    (4,4)*{};(-4,12)*{} **\crv{(4,7) & (-4,9)}?(1);
    (4,12)*{}; (4,20) **\dir{-}?(1)*\dir{};
    (-4,12)*{}; (-4,20) **\dir{-}?(1)*\dir{};
    (4,-4)*{}; (4,4) **\dir{-}?(1)*\dir{};
    (-4,-4)*{}; (-4,4) **\dir{-}?(1)*\dir{};
    (8,-4)*{}; (8,20) **\dir{-}?(1)*\dir{};
\endxy}\;\; + \;\; r(x_i,x_{i+1},x_{i+2})\;\;
\vcenter{\xy
    (-4,4)*{};(4,12)*{} **\crv{(-4,7) & (4,9)}?(1);
    (4,4)*{};(-4,12)*{} **\crv{(4,7) & (-4,9)}?(1);
    (4,12)*{}; (4,20) **\dir{-}?(1)*\dir{};
    (-4,12)*{}; (-4,20) **\dir{-}?(1)*\dir{};
    (4,-4)*{}; (4,4) **\dir{-}?(1)*\dir{};
    (-4,-4)*{}; (-4,4) **\dir{-}?(1)*\dir{};
    (-8,-4)*{}; (-8,20) **\dir{-}?(1)*\dir{};
\endxy}\;\;
\end{equation*}}

The far commutativity relations in Definition \ref{gnhdef} are
pictured by diagrams which show singular crossings or dots on disconnected
strands freely sliding up or down along the vertical axis without interacting.

\begin{theorem}\label{bsgnhthm}
The algebra defined by the presentation in Definition \ref{gnhdef} is isomorphic to the algebra
generated by the operators in Definition \ref{begnhdef}.
\end{theorem}

\begin{proof}

In the universal setting (Theorem \ref{lazardthm}),
each operator $\partial_i$ in Definition \ref{begnhdef} is a $\aP_n$-linear combination of operators in  $\LL[\Sigma_n]$. Since the symmetric group $\Sigma_n$ has the presentation:
$$\Sigma_n = \inp{\s_1,\ldots, \s_{n-1} \,:\, \s_i^2 = 1, \s_i \s_j = \s_j \s_i \normaltext{ for } \vert i-j\vert > 2, \normaltext{ and } \s_i \s_{i+1} \s_i = \s_{i+1} \s_i \s_{i+1}}$$
and $\s_i = (x_i -_F x_{i+1})\partial_i - \Id$, finding a presentation for $\GNH_n$ amounts to a computation. For example, the $\aP_n$-span of the set
$$\{ 1, \partial_i, \partial_{i+1}, \partial_i\partial_{i+1},\partial_{i+1}\partial_i, \partial_i\partial_{i+1}\partial_i, \partial_{i+1}\partial_i\partial_{i+1}\}$$
is linearly dependent and the braid relation is obtained by solving for coefficients in the dependency relation. The commutation relation between $x_{i\pm 1}$ and $\partial_{i}$ can be seen to hold in the same way. The far commutativity relations follow by observing that they hold for polynomials in $x_i$ and the operators $\s_i \in \Sigma_n$ respectively.
\end{proof}

\begin{example}{(deformations of nilHecke algebras)}\label{examplesofgnh}
\begin{enumerate}
\item The genus $\rho$ associated to homology is determined by $\rho(\CC P^{n}) = 0$ and corresponds to the formal group law $F(x,y) = x + y$. For this formal group law the presentation of $\GNH_n$ in Definition \ref{gnhdef} becomes the presentation of $\NH_n$ in Definition \ref{nilheckedef}.

\item For connective K-theory, a graded Todd genus $\Td$, determined by 
$$\Td(\CC P^n) = \beta^n \in   \ZZ[\beta]$$
leads to the formal group law  $F_K(x,y) = x + y - \beta  xy$. This formal group law in turn determines an algebra $\GNH_n$ which is defined over the ring $R = \ZZ[\beta]$ where $\beta$ is placed in degree $-2$. The formula,
$$x_i -_F x_{i+1} = \frac{x_i - x_{i+1}}{1+\beta x_{i+1}},$$
shows that the functions defining $\GNH_n$ are given by:
\[
\begin{array}{lll}
s(x_i,x_{i+1}) = \beta & & t(x_i,x_{i+1}) = 1 - \beta x_{i+1}\\
l(x_i,x_{i+1},x_{i+2}) = 0 & & r(x_i,x_{i+1},x_{i+2}) = 0.
\end{array}
\]
Notice that $l,r = 0$ implies that the braid relation, $\partial_i
\partial_{i+1} \partial_i = \partial_{i+1}\partial_i\partial_{i+1}$ holds.
The formal group laws associated to singular homology and K-theory determine
the only generalized nilHecke algebras $\GNH_n$ for which this is true, see
\cite{BE1}.

\item Related to Hirzebruch's $\chi_a$-genus is the formal group law
$$F_{\chi_{\alpha,\beta}}(x,y) = \frac{x + y - \beta xy}{1 + \alpha xy}$$
which is defined over the graded ring $\ZZ[\alpha,\beta]$ where $\alpha$ is placed
in degree $-4$ and $\beta$ in degree $-2$. The formula,
$$x_i -_F x_{i+1} = \frac{x_i - x_{i+1}}{1-\beta x_{i+1} - \alpha x_i x_{i+1}}$$
shows that the functions defining $\GNH_n$ are given by:
\[
\begin{array}{lll}
s(x_i,x_{i+1}) = \beta & & t(x_i,x_{i+1}) = 1 - \beta x_{i+1} - \alpha x_i x_{i+1}\\
l(x_i,x_{i+1},x_{i+2}) = \alpha & & r(x_i,x_{i+1},x_{i+2}) = -\alpha.
\end{array}
\]
Setting $\beta = 0$ in this example yields a symmetric formal group
law. Notice that the functions $l$ and $r$ do not depend on $\beta$.

Although this genus determines deformations of $\NH_n$, it does not
necessarily correspond to a cohomology theory.

\item There is an algebra $\GNH_n$ defined over
  $\ZZ[\frac{1}{2},\epsilon,\delta]$ using Euler's formal group law in
  Section \ref{complexcomplexsec}. 
\item The universal example is determined by complex cobordism and the universal formal group law. 
\end{enumerate}
\end{example}

\begin{remark}\label{stackrem}
  For each commutative ring $R$, Theorem \ref{lazardthm} implies that the
  set of formal group laws $F\in R[[x,y]]$ is given by $\Hom(\LL,R)$.  The
  \emph{moduli space of formal group laws} is the scheme determined by the
  functor of points $R \mapsto \Hom(\LL,R)$. The construction above yields a
  sheaf of algebras $\aGNH_n$ on this space. The space
  $\Spec(\LL)\cong\AA^{\infty}$ over which this sheaf resides is the
  parameter space for all of the deformations defined in Section
  \ref{nhpresentationssec}.

  The nilHecke algebra $\NH_n$ can be defined as an algebra of
  $\Sym_n$-equivariant endomorphisms of the polynomial ring $\aP_n$. This
  means that the nilHecke algebra has an intrinsic topological definition.
\begin{equation}\label{inteq}
\NH_n = \End_{\Sym_n}(\aP_n) \conj{ suggests } \GNH_n = \End_{h^*(BU(n))}(h^*(BT^n))
\end{equation}

While it is not necessarily true that Definition \ref{gnhdef} agrees with
the definitions suggested above, (although, see Remark \ref{beisstupidrem}),
these more intrinsic definitions show that the sheaves of algebras
associated to them descend from sheaves on the moduli space of formal group
laws to sheaves on the moduli stack of isomorphism classes of formal group
laws. Over a field of characteristic 0, this stack degenerates to a single
point.  This suggests that one should restrict to spaces of
deformation parameters which have fewer automorphisms. Definition
\ref{symdef} may be one step in this direction.
\end{remark}

\section{Deformations with positive gradings}\label{defpossec}

Recently the author was asked whether there are deformations of the nilHecke
algebras with positively graded deformation parameters. In this section a
family of such deformations is defined. The divided difference operators
used to obtain these deformations are related to discrete Fourier
transforms.

\begin{remark}\label{beisstupidrem}
When $F\in R[[x,y]]$ is a graded formal group law, if 
$$t(x_j,x_{j+1})=\frac{x_j - x_{j+1}}{x_j -_F x_{j+1}} = \sum_{r,s} c_{r,s} x_j^r x_{j+1}^s\conj{ where } c_{r,s}\in R$$
then the Bressler-Evens divided difference operator $\partial_j^F$ is related to the usual divided difference operator $\partial_j$ by:
\begin{equation}\label{yaddaeq}
\partial_j^F(f) = \sum_{r,s} c_{r,s} \partial_j(x_j^r x_{j+1}^s f) = \sum_{r,s} d_{r,s} x_j^r x_{j+1}^s \partial_j(f),
\end{equation}
where further application of the Liebniz rule \eqref{leibeq} produces the right hand side.
(Alternatively, the $\Sym_n$-equivariance of $\partial_j^F$ in Proposition \ref{symprop} and Equation \eqref{inteq} imply that, after a completion, such a relation must exist and grading contraints force each term in the series to be the form above.)
\end{remark}

This suggests that in order to deform the nilHecke algebra over the
ground ring $\ZZ[\beta]$ with $|\beta| \in\ZZ_+$ we should look for a family
of operators $\ga_j^\beta : \aP_n\to \aP_n$ on the polynomial ring $\aP_n =
\ZZ[x_1,\ldots,x_n]$ which satisfy $|\b \ga^\b_j| = -2$ and set
$\tilde{\partial}_j = \partial_j + \b \ga^\b_j$ to obtain an operator which
deforms the divided difference operator $\partial_j$. Degree constraints and
the desire for locality leave us with operators of the form $\ga^\b_j(f) =
1/p_j \cdot \d(f)$ where $\d$ is an operator of degree 0, $p_j \in
R[x_j,x_{j+1}]$ and $p_j\,\vert\,\d(f)$ for each $f\in\aP_n$. 

This motivates the definition below.

\begin{definition}
Let $\rho^{/n} = e^{2 \pi i /n}\in \CC$. Then there are operators $\rho_j^{/n} : \aP_n\otimes \CC \to \aP_n \otimes \CC$ which act on $f\in \aP_n$ by
$$\rho_j^{/n}(f(x_1,\ldots,x_j,\ldots,x_n)) = f(x_1,\ldots,\rho^{/n} x_j,\ldots,x_n).$$
Setting $\rho_j^{k/n} = (\rho_j^{/n})^k$ allows us to define the Cartier operator: $\Lambda^{n}_j = \sum_k \rho_j^{k/n}$. Since $\rho_k^{/n} \rho_j^{/n} = \rho_j^{/n} \rho_k^{/n}$, it follows that $\Lambda^n_k \Lambda^n_j = \Lambda^n_j \Lambda^n_k$. Notice that $(x_j^n-x_{j+1}^n)\,\vert\, \Lambda^n_{j} \Lambda^n_{j+1} f$ for all $f\in \aP_n$. 

The \emph{$n$th cyclotomic divided difference operator} $\ga^n_j : \aP_n \to \aP_n$, $0<j<n$, is given by
\begin{equation}\label{gadef}
\ga^{n}_j = (\Id + \s_j) \frac{1}{x_j^n - x_{j+1}^n}\Lambda^n_j\Lambda^n_{j+1}.
\end{equation}
The degree is given by the rule $|\ga^n_j| = -2n$. Divisibility considerations imply that $\ga^n_j$ takes polynomials with integer coefficients to polynomials with integer coefficients.
\end{definition}

\begin{example}
After setting $x = x_1$ and $y= x_2$, we can write $\ga^2_1(f)$ as
$$\frac{f(x,y)-f(y,x)+f(-x,y)-f(-y,x)+f(x,-y)-f(y,-x)+ f(-x,-y)-f(-y,-x)}{x^2-y^2}.$$
\end{example}

In the general case, the numerator of $\ga^n_j$ involves a sum over the
regular $n$-gon in the circle $S^1$ formed by the $n$th roots of unity. This
ensures that the polynomial in the numerator is divisible by the
denominator.

\begin{proposition}
The cyclotomic divided difference operators $\ga^n_j$ satisfy the following properties:
\begin{enumerate}
\item If $f\in \aP_n$ is a polynomial then $\ga^n_j(f)\in \aP_n^{(j,j+1)}$ is a polynomial which is symmetric in the variables $x_j$ and $x_{j+1}$.
\item If $f\in \aP_n^{(j,j+1)}$ is a polynomial which is symmetric in the variables $x_j$ and $x_{j+1}$ then $\ga^n_j(f) = 0$.
\end{enumerate}
\end{proposition}

The proof of the proposition follows from an examination of the form of
Equation \eqref{gadef}. Compare to Proposition \ref{symprop}.

\begin{definition}
Let $\BB = \ZZ[\b_1,\b_2,\ldots]$ be the graded polynomial ring in infinitely many generators $\b_n$ where $|\b_n| = 2n-2$. The \emph{cyclotomic deformation} $\NH_n^{\cyc}$ of the nilHecke algebra is the algebra determined by the action of the operators:
$$\tilde{\partial}_j = \partial_j + \sum_{n\geq 1} \b_n \ga_j^n$$
on the polynomial ring $\aP_n \otimes \BB$. 

The operators $\tilde{\partial}_j$ take polynomials to polynomials because
the sum $\tilde{\partial}_j(f)$ contains only finitely many non-zero terms
for any polynomial $f\in\aP_n$.
\end{definition}

\begin{remark}
A refinement of the operators $\ga_j^n$ is obtained by writing $x^n-y^n = \prod_{d\vert n} \Phi_d(x,y)$ and summing over only those roots associated to a denominator of the form $\prod_i \Phi_{d_i}(x,y)$.
\end{remark}

\begin{remark}
  A computation shows that the operators $\ga_{j}^n$ are not
  $\Sym_n$-equivariant. In particular, none of these deformations are
  determined by oriented cohomology theories. 
\end{remark}

\end{document}

%% file: mypackages.tex
\usepackage{amsmath, amscd, amssymb, amsthm,amsfonts}
\usepackage{stmaryrd,euscript, mathrsfs, latexsym}
\usepackage{color}
\usepackage{hyperref} 

\usepackage{lmodern}
\usepackage[T1]{fontenc} 

\ifx\pdfoutput\undefined
\usepackage{graphicx}
\else
\fi

\ifx\xetex
\usepackage{fontspec}
\usefont{EU1}{lmr}{m}{n}
\else
\fi

\usepackage{tikz}
\usetikzlibrary{matrix,arrows,positioning}
\tikzset{node distance=2cm, auto}

\usepackage[nodayofweek]{datetime}

\parskip=\medskipamount

%% file: mymacros.tex
\DeclareMathOperator{\Td}{Td}

\DeclareMathOperator{\SL}{SL}

\DeclareMathOperator{\img}{im}

\DeclareMathOperator{\Spec}{Spec}

\DeclareMathOperator{\Sym}{Sym}

\DeclareMathOperator{\Morph}{Hom}

\DeclareMathOperator{\inv}{inv}

\DeclareMathOperator{\End}{End}

\DeclareMathOperator{\Hom}{Hom}

\DeclareMathOperator{\NH}{NH}
\DeclareMathOperator{\GNH}{GNH}

\newcommand{\conj}[1]{\quad\textnormal{ #1 }\quad}

\newcommand{\inp}[1]{\ensuremath{\langle #1 \rangle}}

\newcommand{\normaltext}[1]{\textnormal{#1}}

\makeatletter
\def\imod#1{\allowbreak\mkern2.5mu({\operator@font mod}\,#1)}
\makeatother

\renewcommand{\b}{\beta}

\newcommand{\s}{\sigma}

\newcommand{\ga}{\gamma}
\renewcommand{\d}{\delta}
\DeclareMathOperator{\Id}{Id}
\DeclareMathOperator{\Fl}{Fl}

\newcommand{\aP}{\mathcal{P}}

\usepackage{bbm}

\renewcommand{\AA}{\mathbb{A}}
\newcommand{\BB}{\mathbb{B}}
\newcommand{\CC}{\mathbb{C}}

\newcommand{\LL}{\mathbb{L}}

\newcommand{\QQ}{\mathbb{Q}}

\newcommand{\ZZ}{\mathbb{Z}}

\theoremstyle{plain}
\newtheorem{theorem}[subsection]{Theorem}

\newtheorem{proposition}[subsection]{Proposition}
\newtheorem{corollary}[subsection]{Corollary}

\theoremstyle{remark}

\newtheorem{remark}[subsection]{Remark}

\theoremstyle{definition}
\newtheorem{example}[subsection]{Example}
\newtheorem{examples}[subsection]{Examples}

\newtheorem{definition}[subsection]{Definition}

\newtheorem{notation}[subsection]{Notation}

\numberwithin{equation}{section}